\documentclass[12pt]{article}
\usepackage{array,amsmath,amssymb,amsthm}  				    
\usepackage{fullpage,lineno}
\usepackage{enumerate}
\usepackage{tikz}
\usepackage{bm}
\usepackage{color}

\usepackage{graphicx}

\newtheorem{theorem}{Theorem}[section]
\newtheorem{proposition}[theorem]{Proposition}

\newtheorem{lemma}[theorem]{Lemma}

\theoremstyle{definition}

\begin{document}
\title{On sizes of 1-cross intersecting set pair systems}
\author{
Alexandr V. Kostochka\thanks{University of Illinois at Urbana--Champaign,
Urbana IL, and Sobolev Institute of Mathematics, Novosibirsk,
Russia: \texttt{kostochk@math.uiuc.edu}.  Research supported in part by NSF
grant DMS-1600592, NSF RTG Grant DMS-1937241 and grants 18-01-00353A and 19-01-00682  of the Russian
Foundation for Basic Research.}\,,
Grace McCourt\thanks{University of Illinois at Urbana--Champaign, Urbana IL: \texttt{mccourt4@illinois.edu}.
Research supported in part by NSF RTG Grant DMS-1937241.}\,,
Mina Nahvi\thanks{University of Illinois at Urbana--Champaign,
Urbana IL: \texttt{mnahvi2@illinois.edu}.
Research supported in part by Arnold O. Beckman Campus Research Board Award RB20003 of the University of Illinois at Urbana-Champaign.
}
}

\date{\today}
\maketitle

\baselineskip 16pt
\begin{abstract}
Let $\{(A_i,B_i)\}_{i=1}^m$ be a set pair system. F\"{u}redi, Gy\'{a}rf\'{a}s and Kir\'{a}ly called it {\em $1$-cross intersecting} if $|A_i\cap B_j|$ is $1$ when $i\neq j$ and $0$ if $i=j$. They studied such systems and their generalizations, and in particular considered $m(a,b,1)$ --- the maximum size of a $1$-cross intersecting set pair system in which  $|A_i|\leq a$ and $|B_i|\leq b$ for all $i$. F\"{u}redi, Gy\'{a}rf\'{a}s and Kir\'{a}ly  proved that $m(n,n,1)\geq 5^{(n-1)/2}$ and asked whether there are upper bounds on $m(n,n,1)$ significantly better than the classical bound ${2n\choose n}$ of Bollob\' as for 
cross intersecting set pair systems.

Answering one of their questions, Holzman recently proved that if $a,b\geq 2$, then 
$m(a,b,1)\leq \frac{29}{30}\binom{a+b}{a}$. He also conjectured that the factor $\frac{29}{30}$ in his bound can be replaced by
$\frac{5}{6}$. The goal of this paper is to prove this bound.

\end{abstract}
\section{Introduction}
Let $\{(A_i,B_i)\}_{i=1}^m$ be a family of $m\geq2$ pairs of finite sets. This system is \textit{cross intersecting} if 
\begin{center}
    $A_i\cap B_i=\emptyset$ and $A_i\cap B_j\not=\emptyset$ for all distinct $i,j\in [ m]$.
\end{center} 
This notion introduced by Bollob\'{a}s~\cite{Bol} turned out to be a quite useful concept in extremal combinatorics.
 The reader can find interesting results on cross intersecting set pair systems and their applications  in 
 surveys~\cite{Furedi},\cite{Tuza1} and \cite{Tuza2}.
 
The classical result of Bollob\'{a}s~\cite{Bol} on the sizes of such systems is as follows.

\begin{theorem}[\cite{Bol}]\label{tbb} Let $\{(A_i,B_i)\}_{i=1}^m$ be a cross intersecting set pair system with $|A_i|\leq a_i$ and $|B_i|\leq b_i$ for every $1\leq i\leq m$. Then
$$ \sum_{i=1}^m \frac{1}{\binom{a_i+b_i}{a_i}}\leq1,$$
and equality holds only if there exist $a$, $b$ and an $(a+b)$-element set $M$ such that for every $i$, $A_i$ is an $a$-element subset of $M$ and $B_i=M-A_i$.

In particular, if $\{(A_i,B_i)\}_{i=1}^m$ is a cross intersecting set pair system with $|A_i|\leq a$ and $|B_i|\leq b$ for every $1\leq i\leq m$,  then
$$m \leq {\binom{a+b}{a}}.$$
\end{theorem}

F\"{u}redi, Gy\'{a}rf\'{a}s and Kir\'{a}ly~\cite{FGK} have introduced a more restricted class of set pair systems. 
They call a set pair system \textit{$1$-cross intersecting} if it is cross intersecting and $|A_i\cap B_j|=1$ 
for all distinct $i,j\in [ m]$.
 A set pair system $\{(A_i,B_i)\}_{i=1}^m$ is $(a,b)$-\textit{bounded} if for all $i$ we have $|A_i|\leq a$ and $|B_i|\leq b$. 
F\"{u}redi, Gy\'{a}rf\'{a}s and Kir\'{a}ly~\cite{FGK} studied $1$-cross intersecting set pair systems and some variations of them. 
They also pointed out connections of these problems with problems on  edge partitions of special bipartite graphs into complete bipartite subgraphs.

In particular, they considered how large  such systems can be if the sizes of the sets in the systems are bounded.
Let $m(a,b,1)$ denote  the maximum size of a $1$-cross intersecting set pair system in which  $|A_i|\leq a$ and $|B_i|\leq b$ for all $i$. F\"{u}redi, Gy\'{a}rf\'{a}s and Kir\'{a}ly~\cite{FGK} proved that $m(n,n,1)$ is at least exponential in $n$:

\begin{proposition}[\cite{FGK}] \label{FGK-prop1} If $n$ is even, then $m(n,n,1)\geq 5^{n/2}$, and if $n$ is odd, then $m(n,n,1)\geq 2\cdot 5^{(n-1)/2}$.
\end{proposition}

On the other hand, they conjectured that there exists an $\epsilon>0$ such that $m(n,n,1)\leq (1-\epsilon)\binom{2n}{n}$
for every $n\geq 2$ and that 
\begin{equation}\label{con2}
\lim_{n\to \infty}\frac{m(n,n,1)}{\binom{2n}{n}}=0.
\end{equation}

Very recently, Holzman~\cite{Holzman}  proved the first conjecture in the following stronger form.

\begin{theorem}
[\cite{Holzman}]\label{Holz-thm}  Let $a_i, b_i\geq2$ for $1\leq i\leq m$, and let $\{(A_i,B_i)\}_{i=1}^m$ be a 1-cross intersecting set pair system with $|A_i|\leq a_i$ and $|B_i|\leq b_i$ for every $1\leq i\leq m$. Then
 $$   \sum_{i=1}^m \frac{1}{\binom{a_i+b_i}{a_i}}\leq\frac{29}{30}.$$
In particular, if $a,b\geq 2$ and $\{(A_i,B_i)\}_{i=1}^m$ is a 1-cross intersecting set pair system with $|A_i|\leq a$ and $|B_i|\leq b$ for every $1\leq i\leq m$, then
$$  m(a,b,1)\leq\frac{29}{30}{\binom{a+b}{a}}.$$
\end{theorem}

One of the ideas of Holzman was to prove a stronger and more detailed statement in order to employ the stronger induction assumption similar to the ideas of Bollob\'{a}s~\cite{Bol} in the proof of Theorem~\ref{tbb}. Holzman~\cite{Holzman} also writes:

\bigskip
\parbox{6.1in}{\em
It seems likely that our constant $\frac{29}{30}$ can be improved to $\frac{5}{6}$, which would be best possible...
}

\medskip
Mentioning sharpness of  $\frac{5}{6}$, Holzman refers to the following result of F\"{u}redi, Gy\'{a}rf\'{a}s and Kir\'{a}ly:
\begin{proposition}[\cite{FGK}] \label{FGK-prop} Let $\{(A_i,B_i)\}_{i=1}^m$ be a 1-cross intersecting set pair system with $|A_i|\leq 2$ and $|B_i|\leq 2$. Then $m\leq5$, and equality holds only if $\{A_i\}_{i=1}^5$ and $\{B_i\}_{i=1}^5$ form two complementary 5-cycles (that is, the vertices may be written as $0,1,2,3,4\pmod5$, so that $A_i=\{i,i+1\}$ and $B_i=\{i-1,i+2\}$ for $1\leq i\leq5$).
\end{proposition}

The goal of this paper is to confirm Holzman's conjecture:

\begin{theorem}\label{kmn}  Let $a_i, b_i\geq2$ for $1\leq i\leq m$, and let $\{(A_i,B_i)\}_{i=1}^m$ be a 1-cross intersecting set pair system with $|A_i|\leq a_i$ and $|B_i|\leq b_i$ for every $1\leq i\leq m$. Then
  $$  \sum_{i=1}^m \frac{1}{\binom{a_i+b_i}{a_i}}\leq\frac{5}{6}.$$
In particular, if $a,b\geq 2$ and $\{(A_i,B_i)\}_{i=1}^m$ is a 1-cross intersecting set pair system with $|A_i|\leq a$ and $|B_i|\leq b$ for every $1\leq i\leq m$, then
$$  m(a,b,1)\leq\frac{5}{6}{\binom{a+b}{a}}.$$
\end{theorem}

Our proof heavily uses and develops ideas of Holzman~\cite{Holzman}. In particular, instead of Theorem~\ref{kmn} we prove the following slightly stronger statement in order to use the stronger induction assumption.

\begin{theorem}\label{kmn3}
 Let $\{(A_i,B_i)\}_{i=1}^m$ be a 1-cross intersecting set pair system such that $|A_i|=a_i$ and $|B_i|=b_i$ for every $i$, $1\leq i\leq m$. Then
 $$\sum_{i=1}^{m}\frac{1}{\binom{a_i+b_i}{a_i}}\leq\frac{5}{6}$$ unless for some $1\leq i<j\leq m$ one of the following occurs:
 \begin{enumerate}[(a)]
     \item $|A_i|=|A_j|=1$ and $B_i\cap B_j\not=\emptyset$, or
     \item $|B_i|=|B_j|=1$ and $A_i \cap A_j \not= \emptyset$, or
     \item $|A_i|=|A_j|=|B_i|=|B_j|=1$.
 \end{enumerate}
\end{theorem}

The structure of the paper is as follows. In the next section, we introduce notation, discuss the setup, cite two important lemmas 
from~\cite{Holzman}, and prove two new lemmas. In Section 3 we prove Theorem~\ref{kmn3}.


\section{Setup and lemmas}

Following the notation in~\cite{Holzman}, for a set pair system ${\cal{S}}=\{(A_i,B_i)\}_{i\in I}$ with ground set $V({\cal{S}})=\cup_{i\in I}(A_i\cup B_i)$ and some $R\subseteq V({\cal{S}})$, ${\cal{S}}-R$ is the set pair system $\{(A_i\backslash R,B_i\backslash R)\}_{i\in I}$. 

An immediate corollary of this definition is that if ${\cal{S}}$ is 1-cross intersecting and there exists no $v\in R$ and $i\not=j\in I$ such that $v\in A_i\cap B_j$, then ${\cal{S}}-R$ is also 1-cross intersecting. Also, for $J\subseteq I$, ${\cal{S}}[J]$ is the set pair system $\{(A_i,B_i)\}_{i\in J}$.

For a set pair system ${\cal{S}}=\{(A_i,B_i)\}_{i\in I}$, let $\Sigma({\cal{S}})=\sum_{i\in I}\frac{1}{\binom{|A_i|+|B_i|}{|A_i|}}$. For any $v\in V({\cal{S}})$, let $I_{\bar{v}}^A=\{i\in I| v\not\in A_i\}$ and similarly $I_{\bar{v}}^B=\{i\in I| v\not\in B_i\}$. Now, we are ready to cite the following  result from~\cite{Holzman}, where the idea of Bollob\' as' proof of Theorem~\ref{tbb} is stated in a convenient form.
  
\begin{lemma}[\cite{Holzman}] \label{avg}
Let ${\cal{S}}=\{(A_i,B_i)\}_{i\in I}$ be a set pair system such that $A_i\not=\emptyset$, $B_i\not=\emptyset$ and $A_i\cap B_i=\emptyset$ for every $i\in I$. Then
\begin{center}
    $\Sigma({\cal{S}})=\frac{1}{|V({\cal{S}})|}\sum_{v\in V({\cal{S}})}\Sigma({\cal{S}}[I_{\bar{v}}^A]-\{v\})\leq \max_{v\in V({\cal{S}})}\Sigma({\cal{S}}[I_{\bar{v}}^A]-\{v\})$,
\end{center}
and similarly
\begin{center}
    $\Sigma({\cal{S}})=\frac{1}{|V({\cal{S}})|}\sum_{v\in V({\cal{S}})}\Sigma({\cal{S}}[I_{\bar{v}}^B]-\{v\})\leq \max_{v\in V({\cal{S}})}\Sigma({\cal{S}}[I_{\bar{v}}^B]-\{v\})$.
\end{center}
\end{lemma}
We will also use the following observation of Holzman.
\begin{lemma}[\cite{Holzman}] \label{1/3} 
For $a,b\geq2$ we have $\frac{\binom{a+b-2}{a-1}}{\binom{a+b}{a}}\leq\frac{1}{3}$. Moreover, the upper bound can be improved to $\frac{3}{10}$ unless $a=b=2$.
\end{lemma}

Our first lemma is in the spirit  of Lemma \ref{1/3}.
\begin{lemma} \label{1/5}
For $a,b\geq2$, we have $\frac{\binom{a+b-3}{b-1}}{\binom{a+b}{b}}\leq\frac{1}{5}$.
\end{lemma} 
\begin{proof}
	Note that $$\frac{\binom{a+b-3}{b-1}}{\binom{a+b}{b}}=\frac{ab(a-1)}{(a+b)(a+b-1)(a+b-2)}.$$ Let $g(a,b)=\frac{(a+b)(a+b-1)(a+b-2)}{ab(a-1)}$. We want to show $g(a,b)\geq5$, for all integers $a,b\geq2$. Let $c=a-1$. We have $g(a,b)=g(c+1,b)=\frac{(b+c+1)(b+c)(b+c-1)}{bc(c+1)}$. So, 
$$\mbox{\em	$g(a,b)\geq5$  if and only if $f(c,b)=(b+c)^3-(b+c)-5bc(c+1)\geq0$.}$$
 Now, the derivative of $f$ with respect to $c$ is $$3(b+c)^2-1-5b(2c+1)=3(b-c)^2+b(2c-5)-1,$$ which is positive for $c\geq 3$ and $b\geq2$. On the other hand, $$f(3,b)=b^3+9b^2-34b+24=(b-1)(b-2)(b+12)\geq0$$ for $b\geq2$, which together with the positive derivative proves $g(a,b)\geq5$ for $a\geq4$ and $b\geq2$. Now, when $c=1$ we have $f(1,b)=b^3+3b^2-8b>0$ for $b\geq2$, implying $g(2,b)\geq5$ for $b\geq2$. When $c=2$ and therefore $a=3$, we have $f(2,b)=b^3+6b^2-19b-2>0$ for $b\geq3$, and $g(3,2)=5$, which implies $g(3,b)\geq5$ for $b\geq2$. This finishes the proof.
\end{proof}

For ease of notation, define $f(A_i, B_i) = \frac{1}{\binom{|A_i|+|B_i|}{|A_i|}}$ for any two sets $A_i, B_i$. Now, we prove the following Lemma.

\begin{lemma}\label{22-Diamond}
In a 1-cross intersecting set pair system ${\cal{S}} = \{(A_i, B_i)\}_{i \in I}$, if $|A_i| = |A_j| = 2$ with $A_i \cap A_j \neq \emptyset$ and $|B_i| = |B_j| = 2$ with $B_i \cap B_j \neq \emptyset$ for some $i, j \in I$, then $\Sigma({\cal{S}}) \leq \frac{5}{6}$.
\end{lemma}

\begin{proof}
Suppose we have $|A_1| = |A_2| = 2$ with $A_1 \cap A_2 \neq \emptyset$ and $|B_1| = |B_2| = 2$ with $B_1 \cap B_2 \neq \emptyset$. Say $A_1 = \{v,x\}$, $A_2 = \{v,y\}$, $B_1 = \{y,w\}$, and $B_2 = \{x,w\}$. Notice that if $x \in A_i$ for some $i \geq 3$, then we must also have $y \in A_i$ in order for $A_i$ to intersect $B_1$. Similarly, if $x \in B_j$ for $j \geq 3$, then $y \in B_j$. Thus we must have either $x,y \in A_i$ or $w \in A_i$, and either $x,y \in B_j$ or $v \in B_j$. Since our set pair system is 1-cross intersecting, we cannot have $x,y \in A_i$ and $x,y \in B_j$. So, by symmetry, we may assume $x,y \notin B_j$ for all $j \geq 3$. Notice that this gives $v \in B_j$ for all $j \geq 3$. 

We partition all indices other than 1 and 2 into two sets. Let $I_1 = \{i \geq 3 \, | \, w \in B_i\}$, $I_2 = \{i \geq 3 \, | \, w \notin B_i\}$. Notice that for all $i,j \in I_1$, $x,y \in A_i$ and $v,w \in B_j$, so no $A_i$, $B_j$ can intersect in any of $v,w,x,y$. Thus ${\cal{S}}[I_1] - \{v,w,x,y\}$ is 1-cross intersecting, so $\Sigma({\cal{S}}[I_1] - \{v,w,x,y\}) \leq 1$. 

{\bf Case 1:}  $|A_i|, |B_i| \geq 2$ for all $i \in I$. 
 If there exists $i \in I_1$ with $|A_i|=|B_i| = 2$, then $A_1, A_2, A_i$ form a triangle, so $i = 3$ and there are no other sets in our system. This gives $$\Sigma({\cal{S}}) = f(A_1, B_1) + f(A_2, B_2) + f(A_i, B_i) = 1/6 +1/6 + 1/6 < 5/6,$$ and we are done. Thus we may assume that there are no such pairs in $I_1$. Applying Lemma \ref{1/3} twice gives
\[ \Sigma({\cal{S}}[I_1]) = \sum\limits_{i \in I_1} \frac{1}{\binom{|A_i|+|B_i|}{|A_i|}} \leq \frac{3}{10}\sum\limits_{i \in I_1} \frac{1}{\binom{|A_i|+|B_i|-2}{|A_i|-1}} \leq \frac{1}{10} \sum\limits_{i \in I_1} \frac{1}{\binom{|A_i|+|B_i|-4}{|A_i|-2}} \leq \frac{1}{10}. \]

Since for all $i, j \in I_2$, $v \in B_j$, $w,x,y \notin B_j$, and $v \notin A_i$, we also have that ${\cal{S}}[I_2] - \{v,w,x\}$ is 1-cross intersecting, so $\Sigma({\cal{S}}[I_2] - \{v,w,x\}) \leq 1$. Notice also that $|A_i - \{v,w,x\}| = |A_i| - 1$ and $|B_i - \{v,w,x\}| = |B_i|-1$ for all $i \in I_2$. We can apply Lemma \ref{1/3} to get  
\[ \Sigma({\cal{S}}[I_2])=\sum\limits_{i \in I_2} \frac{1}{\binom{|A_i|+|B_i|}{|A_i|}} \leq \frac{1}{3}\sum\limits_{i \in I_2} \frac{1}{\binom{|A_i|+|B_i|-2}{|A_i|-1}} \leq \frac{1}{3}.\]

Thus in total, we have $$ \Sigma({\cal{S}}) = f(A_1, B_1) + f(A_2, B_2) + \Sigma({\cal{S}}[I_1]) + \Sigma({\cal{S}}[I_2]) \leq \frac{1}{6} + \frac{1}{6} + \frac{1}{10} + \frac{1}{3} = \frac{23}{30} < \frac{5}{6}. $$

{\bf Case 2:} ${\cal{S}}$ contains a set of size $1$,  say $|A_3|=1$. Then we must have $A_3 = \{w\}$ and $w \in B_i$ for all $i \geq 4$. Thus $I_2 = \{3\}$ and $I_1 = I - [3]$. Recall that for all $i \in I_1$, $x,y \in A_i$ and $v,w \in B_i$. If $|B_3| = 1$, then $I_1 = \emptyset$ since $B_3 = \{v\}$ and no $B_i$ can contain $B_3$. So, in this case 
$$ \Sigma({\cal{S}}) = f(A_1, B_1) + f(A_2, B_2) + f(A_3,B_3)=\frac{1}{6} + \frac{1}{6} + \frac{1}{2} = \frac{5}{6}. $$

If $|B_3| \geq 2$, then $f(A_3,B_3) \leq \frac{1}{3}$ and we may have $I_1 \neq \emptyset$. By the same argument as above, $\Sigma({\cal{S}}[I_1]) \leq \frac{1}{10}$ in this case. Thus in total, we have 
$$ \Sigma({\cal{S}}) = f(A_1, B_1) + f(A_2, B_2) + f(A_3, B_3) + \Sigma({\cal{S}}[I_1])\leq \frac{1}{6} + \frac{1}{6} + \frac{1}{3} + \frac{1}{10} = \frac{23}{30} < \frac{5}{6}. $$ This concludes all possible cases.
\end{proof}

\section{Proof of Theorem~\ref{kmn}}

Suppose that the theorem does not hold and
	 ${\cal{S}}=\{(A_i,B_i)\}_{i\in I}$ is a counter-example  is minimal with respect to $\sum_{i\in I} (a_i+b_i)$. 
Since adding to a set $A_i$ or $B_i$ an element outside of $V({\cal{S}})$ leaves the system $1$-cross intersecting,
we may assume that $|A_i|=a_i$ and $|B_i|=b_i$ for all $i\in I$.
	 
	 Let $I=\{1,2,\ldots,m\}$. By the choice of ${\cal{S}}$, $\Sigma({\cal{S}})=\sum_{i\in I}\frac{1}{\binom{a_i+b_i}{a_i}}>\frac{5}{6}$. We consider cases based on the number of sets $A_i$ or $B_j$ with size 1.
	 
\medskip	 
	
\noindent{\bf Case 1:} There is exactly one set of size $1$.
		Without loss of generality, let $|A_1|=1$ and $A_1=\{x\}$. We have two subcases:
		
        \emph{Subcase 1.1:} There exists some $ i\in I-\{1\}$ such that $|B_i|=2$.
	We can assume $|B_2|=2$. Since $|A_1\cap B_2|=1$, we may assume $B_2=\{x,y\}$. This means for each $j\geq3$,  $y\in A_j$ and $x\in B_j$. Then by deleting $x$ and $y$ from all $A_j$ and $B_j$, $j \geq 3$, we get for $I_1=I-\{1,2\}$
 a 1-cross intersecting system ${\cal{S}}_1=\{(A_i-y,B_i-x)\}_{i\in I_1}$.
	 So, by Lemmas \ref{avg} and \ref{1/3},
			\[\Sigma({\cal{S}}[I_1])=\sum_{i\in I_1}\frac{1}{\binom{|A_i|+|B_i|}{|A_i|}}\leq \frac{1}{3}\sum_{i\in I_1}\frac{1}{\binom{|A_i|+|B_i|-2}{|A_i|-1}}=\frac{1}{3}\Sigma({\cal{S}}[I_1]-\{x,y\})\leq\frac{1}{3}. \] This implies $\Sigma({\cal{S}}) = f(A_1, B_1)+f(A_2,B_2) + \Sigma({\cal{S}}[I_1])\leq \frac{1}{3}+\frac{1}{6}+\frac{1}{3}=\frac{5}{6}.$
			
			\emph{Subcase 1.2:} For all $i\geq2$,  $|B_i|\geq3$.
By Lemma \ref{avg} we know there exists some $v\in V({\cal{S}})$ such that $\Sigma({\cal{S}}[I_{\bar{v}}^A]-\{v\})>\frac{5}{6}$. However, since $|B_i|\geq3$ for all $i\geq2$, ${\cal{S}}[I_{\bar{v}}^A]-\{v\}$ either has at most one set of size $1$, or it has exactly two sets of size 1, which can only occur if $|B_1|=2$ and $v \in B_1$.  In the former scenario, by the minimality of ${\cal{S}}$
 we reach a contradiction.
In the latter scenario, suppose $B_1=\{v,y\}$. Now, all the $A_k$ sets in ${\cal{S}}[I_{\bar{v}}^A]-\{v\}$ have $y$ in them, and all the $B_k$ sets in ${\cal{S}}[I_{\bar{v}}^A]-\{v\}$ have $x$ in them (where $k\geq2$), which means for $J=I-\{1\}$ we have 
$$\Sigma({\cal{S}}[I_{\bar{v}}^A]-\{v\})\leq f(A_1, B_1)+\frac{1}{3}\Sigma({\cal{S}}[J_{\bar{v}}^A]-\{v,x,y\})\leq \frac{1}{3}+ \frac{1}{3}=\frac{2}{3}.$$
 a contradiction. 
	
	\bigskip		
\noindent{\bf  Case 2:} There are exactly two sets of size $1$.
		
			\emph{Subcase 2.1:} These two sets are $A_i$ and $B_j$ where $i\not=j$.
			Without loss of generality, let $i=1, j=2$. Since $|A_1 \cap B_2| =1$, we have $A_1=B_2=\{x\}$, which means $x\in A_j,B_j$ for all $j\geq3$. However, if ${\cal{S}}$ has size at least $3$, then we would have $A_3\cap B_3\not=\emptyset$, a contradiction. So, ${\cal{S}}$ has size $2$, and 
$$\Sigma({\cal{S}}) = f(A_1, B_1) + f(A_2, B_2) \leq \frac{1}{3}+\frac{1}{3}=\frac{2}{3}<\frac{5}{6}.$$
			
			\emph{Subcase 2.2:} These two sets are $A_i$ and $B_i$ for some $i\in I$.	
Let $A_1=\{x\}$ and $B_1=\{y\}$, where $x\not=y$. Then, for any $j\geq2$, we have $x\in B_j$ and $y\in A_j$. By an argument very similar to Case 1.1, we obtain a contradiction.
			
			\emph{Subcase 2.3:} These two sets are $A_i$ and $A_j$ where $i \neq j$.
Let $A_1=\{x\}$ and $A_2=\{y\}$. By an assumption in the statement of the theorem, $B_1\cap B_2=\emptyset$. We will have some cases based on the size of $B_1$ and $B_2$.
				
{\em Subcase 2.3.1:}	  $|B_1|=|B_2|=2$.	We have $B_1=\{y,u\}$ and $B_2=\{x,z\}$, where these four vertices are all distinct. Furthermore, for any $i\geq3$ we have $x,y\in B_i$ and $u,z\in A_i$, implying that $|A_i|,|B_i|\geq3$. Let $J=I-\{1,2\}$. By Lemmas \ref{avg} and \ref{1/3} we have
					\begin{align*}
					\Sigma({\cal{S}}) &=f(A_1, B_1) + f(A_2, B_2)+\Sigma({\cal{S}}[J])\leq \frac{2}{3}+\frac{3}{10}\Sigma({\cal{S}}[J]-\{x,u\}) \\
					&\leq \frac{2}{3}+\frac{3}{10}\cdot \frac{1}{3}\Sigma({\cal{S}}[J]-\{x,y,u,z\}) \leq\frac{2}{3}+\frac{1}{10}=\frac{23}{30}<\frac{5}{6}.
					\end{align*}
					
{\em Subcase 2.3.2:}	 				
  $|B_1|=2$ and $|B_2|\geq3$.
					Let $B_1=\{y,u\}$. Recall that $B_1$ and $B_2$ are disjoint. Now, for every $i\geq3$ we have $x,y\in B_i$ and $u\in A_i$. By Lemma \ref{1/5}, for $J=I-\{1,2\}$ we have 
					\[\Sigma({\cal{S}}) = f(A_1, B_1) + f(A_2, B_2)+\Sigma({\cal{S}}[J])\leq \frac{1}{3}+\frac{1}{4}+\frac{1}{5}\Sigma({\cal{S}}[J]-\{x,y,u\})\leq\frac{7}{12}+\frac{1}{5}=\frac{47}{60}<\frac{5}{6}. \]

{\em Subcase 2.3.3:}	 	 $|B_1|,|B_2|\geq3$.
If there exists a $B_i$ of size 2, say $|B_3|=2$, then we have $B_3=\{x,y\}$, which implies $I=[3]$. So, we will have 
$$\Sigma({\cal{S}})= f(A_1, B_1) + f(A_2, B_2) + f(A_3, B_3) \leq\frac{1}{4}+\frac{1}{4}+\frac{1}{6}=\frac{2}{3}<\frac{5}{6}.$$

{\em Subcase 2.3.4:}	  $|B_i|\geq3$, for all $i\in I$. By Lemma \ref{avg} we know there exists some $v\in V({\cal{S}})$ such that $\Sigma({\cal{S}}[I_{\bar{v}}^A]-\{v\})>\frac{5}{6}$. However, ${\cal{S}}[I_{\bar{v}}^A]-\{v\}$ has at most two sets of size $1$
(namely, $A_1$ and $A_2$) because all the sets $B_i$ have size at least $3$. Since  $B_1\cap B_2=\emptyset$, this contradicts the minimality of ${\cal{S}}$.
			
\bigskip
\noindent{\bf  Case 3:} There are at least three sets of size 1.
Repeating the argument of Cases 2.1 and 2.2, we can assume all sets of size 1 in ${\cal{S}}$ are $A_i$s. Now, let $A_i=\{x_i\}$ for $1\leq i\leq3$. Note that $x_3\in B_1,B_2$, which means $B_1$ and $B_2$ are not disjoint, contradicting the assumption in the statement.

\bigskip	
\noindent{\bf  Case 4:} There are no sets of size 1.
We have $a_i,b_i\geq2$ for all $i\in I$. By Lemma \ref{avg}, for some $u,v\in V({\cal{S}})$, we have $\Sigma({\cal{S}}[I_{\bar{v}}^A]-\{v\}),\Sigma({\cal{S}}[I_{\bar{u}}^B]-\{u\})>\frac{5}{6}$. By the minimality of ${\cal{S}}$,  any such ${\cal{S}}[I_{\bar{v}}^A]-\{v\}$ and ${\cal{S}}[I_{\bar{u}}^B]-\{u\}$ both have at least two sets of size $1$. Without loss of generality, we may assume $I_{\bar{u}}^B = \{1, 2, \dots, m'\}$ for some $m' < m$. Let $A'_i := A_i   -  \{u\}$ for all $i \in [m']$. 
Since $|B_i|\geq 2$ for all $i$,
again by the minimality of ${\cal{S}}$, there are $1\leq i<j\leq m'$ such that $|A'_i|=|A'_j|=1$ and $B_i\cap B_j\not=\emptyset$. Say $|A'_1|=|A'_2|=1$. We consider cases based on $|B_1|, |B_2|$.

    \emph{Subcase 4.1:} $|B_1|=|B_2|=2$.
    Since $|A'_1|=|A'_2|=1$,  we have $|A_1|=|A_2|=2$ with $A_1 \cap A_2 = \{u\}$. We also have $B_1\cap B_2\neq \emptyset$, so Lemma \ref{22-Diamond} gives that $\Sigma({\cal{S}})\leq \frac{5}{6}$.
    
    \emph{Subcase 4.2:} $|B_1|=2, |B_2|\geq 3$.
    We have $A'_1 = \{x\}$, $A'_2 = \{y\}$, and $B_1 = \{y,z\}$. For every $3 \leq i \leq m'$, we have $x,y \in B_i$ and $z \in A'_i$. Let $J = \{3, \dots, m'\}$ and notice that ${\cal{S}}[J] - \{x,y,z\}$ is 1-cross intersecting. Applying Lemma \ref{1/5}, we have 
    \[\Sigma({\cal{S}}[I_{\bar{u}}^B]-\{u\}) = f(A'_1, B_1) + f(A'_2, B_2) + \Sigma({\cal{S}}[J]) \leq \frac{1}{3} + \frac{1}{4} + \frac{1}{5}\Sigma({\cal{S}}[J]-\{x,y,z\}) \leq \frac{7}{12} + \frac{1}{5} < \frac{5}{6}. \]
    
    \emph{Subcase 4.3:} $|B_1|, |B_2| \geq 3$.
    Observe that we must also be in the corresponding case for ${\cal{S}}[I_{\bar{v}}^A]-\{v\}$. Let $B^*_i = B_i   -  \{v\}$ for all $i \in I_{\bar{v}}^A$. Then we have $|B^*_i|=|B^*_j| = 1$, $A_i \cap A_j \neq \emptyset$, and $|A_i|, |A_j| \geq 3$ for some $i, j \in I_{\bar{v}}^A$. Recall that $|B_1|, |B_2| \geq 3$, so $i,j$ cannot be 1 or 2. Thus we may assume without loss of generality that 
    $i =3$, $j=4$. Then in ${\cal{S}}$, we have $|A_1| = |B_3| = 2$ and $|B_1|, |A_3| \geq 3$, so we can repeat and develop the corresponding case in Holzman's proof~\cite{Holzman}. 
    
    To match notation, we let $A_1 = \{x,y\}$ and $B_3 = \{x,z\}$, noting that $u = x$ or $y$ and $v = x$ or $z$. We partition $I-\{1,3\}$ into $I_1, I_2, I_3$ as follows:
    \begin{center}
        $I_1=\{i\in I\ :\ x\in A_i, y\in B_i, z\not\in A_i\}$\\
        $I_2=\{i\in I\ : \ x\in B_i, y\not\in B_i, z\in A_i\}$\\
        $I_3=\{i\in I\ : \ x\not\in A_i\cup B_i, y\in B_i, z\in A_i\}$.
    \end{center}
    
    Note that $$\Sigma({\cal{S}}) = f(A_1, B_1) + f(A_3, B_3) + \Sigma({\cal{S}}[I_1]) + \Sigma({\cal{S}}[I_2 \cup I_3]).$$ We have $f(A_1, B_1), f(A_3, B_3) \leq \frac{1}{10}$ and by Lemma \ref{1/3}, $\Sigma({\cal{S}}[I_1]), \Sigma({\cal{S}}[I_2 \cup I_3]) \leq \frac{1}{3}$. Notice that if either ${\cal{S}}[I_1]$ or ${\cal{S}}[I_2 \cup I_3]$ contains no pair of size (2,2), then by Lemma \ref{1/3} we have 
    $$\Sigma({\cal{S}}) \leq \frac{1}{10} + \frac{1}{10} + \frac{3}{10} + \frac{1}{3} = \frac{5}{6}.$$ Thus we may assume ${\cal{S}}[I_1], {\cal{S}}[I_2 \cup I_3]$ each contain at least one pair of size (2,2).
    
    Suppose $|A_i| = |A_j| = |B_i| = |B_j| = 2$ for some $i, j \in I_1$. Notice that since we have $x \in A_i, A_j$ and $y \in B_i, B_j$, Lemma \ref{22-Diamond} gives that $\Sigma({\cal{S}}) \leq \frac{5}{6}$.
    
    Thus we have exactly one such pair in $I_1$, say $(A_i, B_i)$. Consider ${\cal{S}}[I_1] - \{x,y\}$, and notice that $|A_i  - \{x\}| = |B_i   -  \{y\}| = 1$. Then we have 
    \[ \Sigma({\cal{S}}[I_1]-\{x,y\}) = f(A_i  - \{x\}, B_i   -  \{y\}) + \sum_{j \in I_1-\{i\}} f(A_j  - \{x\}, B_j   -  \{y\}) \leq 1.\]
    Since $f(A_i  - \{x\}, B_i   -  \{y\}) = \frac{1}{2}$, this gives $\sum_{j \in I_1-\{i\}} f(A_j  - \{x\}, B_j   -  \{y\}) \leq \frac{1}{2}.$ Then using Lemma \ref{1/3} gives 
    \[ \Sigma({\cal{S}}[I_1]) = f(A_i, B_i) + \sum_{j \in I_1-\{i\}} f(A_j, B_j) \leq \frac{1}{3} \cdot \frac{1}{2} +  \frac{3}{10} \cdot \frac{1}{2} = \frac{19}{60}. \]
    
    If we also have exactly one pair of size (2,2) in $I_2 \cup I_3$, then we apply a similar argument to above to obtain 
    $$\Sigma({\cal{S}}) \leq \frac{1}{10} + \frac{1}{10} + \frac{19}{60} + \frac{19}{60} = \frac{5}{6}.$$ Thus we have at least two such pairs in $I_2 \cup I_3$.
    
    \begin{figure}[h]
		\centering
		\begin{tikzpicture}[scale=1]
		
		    \draw (1,0) node[fill=white, label=below:{\textcolor{red}{\bm{$A_k$}}}] {};
		    \draw (-.25,1) node[fill=white, label=left:{\textcolor{red}{\bm{$A_j$}}}] {};
		    \draw (.25,2.5) node[fill=white, label=above left:{\textcolor{red}{\bm{$A_1$}}}] {};
		    \draw (1.75,2.5) node[fill=white, label=above right:{\textcolor{red}{\bm{$A_i$}}}] {};
		    \draw (2.5,1) node[fill=white, label=right:{\textcolor{red}{\bm{$A_3$}}}] {};
		    
			\draw[color=red, thick] (2,0)--(0,0)--(-.5,2)--(1,3)--(2.5,2);
			\draw[color=blue, thick] (0,0)--(1,3)--(2,0)--(-.5,2)--(2.5,2);
			\draw[color=red, thick]  (2.25,1) circle [x radius=1.2cm, y radius=3mm, rotate=78];
			\draw[color=blue, thick]  (1.25,1) circle [x radius=1.75cm, y radius=3mm, rotate=39];

			\draw[color=black] (0,0) node[draw,shape=circle,fill=black,scale=0.35,label=below left:{\bm{$z$}}] {};
			\draw[color=black] (-0.5,2) node[draw,shape=circle,fill=black,scale=0.35,label=left:{\bm{$y$}}] {};
			\draw[color=black] (1,3) node[draw,shape=circle,fill=black,scale=0.35,label=above:{\bm{$x$}}] {};
			\draw[color=black] (2.5,2) node[draw,shape=circle,fill=black,scale=0.35,label=right:{\bm{$a$}}] {};
			\draw[color=black] (2,0) node[draw,shape=circle,fill=black,scale=0.35,label=below right:{\bm{$b$}}] {};

		\end{tikzpicture}
		\caption{The subsystem $S[\{1,3,i,j,k\}]$.}
	\end{figure}
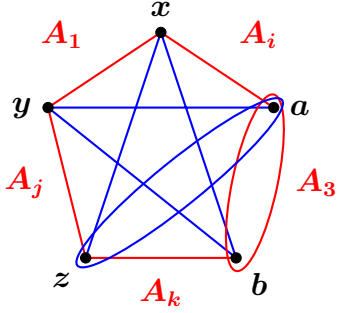

    If either $I_2$ or $I_3$ contains two such pairs, then we can apply Lemma \ref{22-Diamond} to get $\Sigma({\cal{S}}) \leq \frac{5}{6}$. The remaining possibility is that we have exactly one (2,2) pair in each of $I_1$, $I_2$, and $I_3$. We call these pairs $(A_i, B_i)$, $(A_j, B_j)$, and $(A_k, B_k)$, respectively. Recall that $A_1 = \{x,y\}$ and $B_3 = \{x, z\}$. By the definition of $I_1$, we have $A_i = \{x, a\}$ and $B_i = \{y, b\}$ where $a, b, x, y$ and $z$ are all distinct. Notice that using the definition and the fact that $A_k$ and $B_k$ must respectively intersect $B_i$ and $A_i$, we get $A_k = \{z, b\}$ and $B_k = \{y , a\}$. Then in order for $A_j$ to intersect $B_i$ and $B_j$ to intersect $A_k$, we must have $A_j = \{z, y\}$ and $B_j = \{x, b\}$. Finally, we need $a, z \in B_1$ and $a, b \in A_3$ for ${\cal{S}}$
    to be 1-cross intersecting.

    Finally, consider $A_2$. We know $|A_2| = 2$ and $|B_2| \geq 3$, so $i,j,k \neq 2$. However, there is no way for $A_2$ to intersect each of $B_1, B_3, B_i, B_j, B_k$ in exactly one vertex. Thus this case cannot occur, finishing the proof 
    of Theorem~\ref{kmn}.


\end{document}